\documentclass[12pt]{amsart}
\usepackage{amsmath} 
\usepackage{amsfonts} 
\usepackage{amssymb} 
\usepackage{amsthm} 
\usepackage{amscd} 
\usepackage{enumerate}
\usepackage{graphicx,xcolor}
\usepackage{amsopn}
\numberwithin{equation}{section}
\usepackage{wasysym} 

\newcommand{\disp}{\displaystyle}

\renewcommand{\.}{{}_{\!}} 

\renewcommand{\a}{\alpha}

\newcommand{\m}{\mu}

\newcommand{\om}{\omega}

\renewcommand{\r}{\rho}

\renewcommand{\t}{\theta}

\newcommand{\cB}{\mathcal{B}}
\newcommand{\cC}{\mathcal{C}}
\newcommand{\cD}{\mathcal{D}}

\newcommand{\cP}{\mathcal{P}}
\newcommand{\cQ}{\mathcal{Q}}

\newcommand{\cS}{\mathcal{S}}
\newcommand{\cT}{\mathcal{T}}

\newcommand{\as}{\ \mbox{\raisebox{.085ex}{$:$}\!$=$} \ } 

\let\oldexists\exists
\renewcommand{\exists}{\oldexists \,} 
\let\oldforall\forall
\renewcommand{\forall}{\oldforall \,} 

\newcommand{\st}{~|~} 

\newcommand{\ie}{\emph{i.\,e.\,}} 

\newcommand{\inc}{\subseteq} 
\newcommand{\cart}{\! \times \!} 

\renewcommand{\to}{\longrightarrow} 
 
\newcommand{\seqN}[2]{\left( #1_{#2} \right)_{\! #2 \in \NN}} 
\newcommand{\seq}[3]{\left( #1 \right)_{\! #2 \geq #3}} 

\newcommand{\NN}{\mathbf{N}} 
\newcommand{\ZZ}{\mathbf{Z}} 
\newcommand{\RR}{\mathbf{R}} 

\newcommand{\Rn}[1]{\mathbf{R}^{\! #1 \!}} 
\newcommand{\Bn}[1]{\mathbf{B}^{\. #1 \!}} 
\newcommand{\Sn}[1]{\mathbf{S}^{\. #1 \!}} 
 
\renewcommand{\leq}{\leqslant} 
\renewcommand{\geq}{\geqslant} 

\DeclareMathOperator{\atanh}{atanh} 

\let\oldint\int
\renewcommand{\int}[4]{\oldint_{\! #1}^{#2}{\!\!\!\! #3 \mathrm{d} #4}} 

\newcommand{\BB}{\mathrm{B}} 
 
\newcommand{\goes}{\rightarrow} 
 
\newcommand{\vect}[1]{\overrightarrow{#1}} 
 
\let\olddet\det
\renewcommand{\det}[1]{\olddet{\! \left( #1 \right)}} 
 
\newcommand{\GL}{\mathrm{G \. L}} 

\newcommand{\norm}[1]{\left\| #1 \right\|} 

\newcommand{\scal}[2]{\left\langle #1 , #2 \right\rangle} 

\newcommand{\vol}{\mathrm{vol}} 

\theoremstyle{definition}

\newtheorem*{definition*}{Definition} 

\theoremstyle{plain}

\newtheorem{lemma}{Lemma}[section]
\newtheorem*{lemma*}{Lemma} 

\newtheorem{proposition}{Proposition}[section]
\newtheorem*{proposition*}{Proposition} 

\newtheorem{theorem}{Theorem}[section]
\newtheorem*{theorem*}{Theorem} 
\newtheorem*{maintheorem}{Main Theorem}

\newtheorem*{corollary*}{Corollary} 

\newtheorem*{consequence*}{Consequence} 

\newtheorem*{conjecture*}{Conjecture} 

\theoremstyle{definition}

\newtheorem*{notations}{Notations} 

\newtheorem*{remark*}{Remark} 
\newtheorem*{remarks*}{Remarks} 

\newtheorem*{example*}{Example} 
\newtheorem*{examples*}{Examples} 

\newtheorem*{question*}{Question} 
\newtheorem*{questions*}{Questions} 

\newtheorem*{exercise*}{Exercise} 
\newtheorem*{exercises*}{Exercises} 

\newtheorem*{acknowledgment*}{Acknowledgment} 
\newtheorem*{acknowledgments*}{Acknowledgments} 

\theoremstyle{definition}

\newtheorem*{déf*}{Définition} 

\theoremstyle{plain}

\newtheorem*{lem*}{Lemme} 

\newtheorem*{prop*}{Proposition} 

\newtheorem*{thm*}{Théorème} 

\newtheorem*{cor*}{Corollaire} 

\newtheorem*{csq*}{Conséquence} 

\newtheorem*{conj*}{Conjecture} 

\theoremstyle{definition}

\newtheorem*{rem*}{Remarque} 
\newtheorem*{rems*}{Remarques} 

\newtheorem*{ex*}{Exemple} 
\newtheorem*{exs*}{Exemples} 

\newtheorem*{que*}{Question} 
\newtheorem*{ques*}{Questions} 

\newtheorem*{exo*}{Exercice} 
\newtheorem*{exos*}{Exercices} 

\newtheorem*{merci*}{Remerciement} 
\newtheorem*{mercis*}{Remerciements} 

\theoremstyle{remark}

\newcommand{\bC}{\partial \cC} 

\newcommand{\bD}{\partial \cD}
\newcommand{\cTk}{\cT_{k}}

\newcommand{\dC}{d_{\cC \.}} 

\newcommand{\dD}{d_{\cD \.}}
\newcommand{\dP}{d_{\cP \.}}
\newcommand{\dQ}{d_{\. \cQ \.}}
\newcommand{\dB}{d_{\Bn{2}}}
\newcommand{\dT}{d_{\cT \.}}

\newcommand{\FC}{F_{\. \cC \.}} 

\newcommand{\BC}{\cB_{\cC \.}} 

\newcommand{\BP}{\cB_{\cP \.}}
\newcommand{\BQ}{\cB_{\. \cQ \.}}
\renewcommand{\BB}{\cB_{\Bn{2}}}
\newcommand{\BT}{\cB_{\cT \.}}

\newcommand{\SC}{\cS_{\cC \.}} 

\newcommand{\ST}{\cS_{\cT \.}}

\newcommand{\mC}{\m_{\cC \.}} 

\newcommand{\mD}{\m_{\cD \.}}
\newcommand{\mP}{\m_{\cP \.}}
\newcommand{\mQ}{\m_{\. \cQ \.}}
\newcommand{\mB}{\m_{\Bn{2}}}
\newcommand{\mT}{\m_{\cT \.}}

\begin{document}

\title[]{Two properties of volume growth entropy \\ in Hilbert geometry}

\author{Bruno Colbois}
\address{Bruno Colbois, 
Universit\'{e} de Neuch\^{a}tel, 
Institut de math\'{e}matique, 
Rue \'{E}mile Argand 11, 
Case postale 158, 
CH--2009 Neuch\^{a}tel, 
Switzerland}
\email{bruno.colbois@unine.ch}

\author{Patrick Verovic}
\address{Patrick Verovic, 
UMR 5127 du CNRS \& Universit\'{e} de Savoie, 
Laboratoire de math\'{e}matique, 
Campus scientifique, 
73376 Le Bourget-du-Lac Cedex, 
France}
\email{verovic@univ-savoie.fr}

\date{\today}
\subjclass[2000]{Primary: global Finsler geometry, Secondary: convexity}


\begin{abstract}
The aim of this paper is to provide two examples in Hilbert geometry which show 
that volume growth entropy is \emph{not} always a limit on the one hand, 
and that it may vanish for a \emph{non}-polygonal domain in the plane on the other hand. 
\end{abstract}

\maketitle

\bigskip
\bigskip


\section{Introduction} 

A \emph{Hilbert domain} in $\Rn{m}$ is a metric space $(\cC , \dC)$, where $\cC$ is an 
\emph{open bounded convex} set in $\Rn{m}$ and $\dC$ is the distance function on $\cC$ 
--- called the \emph{Hilbert metric} --- defined as follows. 

\medskip

Given two distinct points $p$ and $q$ in $\cC$, 
let $a$ and $b$ be the intersection points of the straight line defined by $p$ and $q$ 
with $\bC$ so that $p = (1 - s) a + s b$ and $q = (1 - t) a + t b$ with $0 < s < t < 1$. 
Then 
$$
\dC(p , q) \as \frac{1}{2} \ln{\! [a , p , q , b]}~,
$$ 
where 
$$
[a , p , q , b] \as \frac{1 - s}{s} \! \times \! \frac{t}{1 - t} \ > \ 1
$$ 
is the cross ratio of the $4$-tuple of ordered collinear points $(a , p , q , b)$ 
(see Figure~\ref{fig:Hilbert-metric}). \\ 
We complete the definition by setting $\dC(p , q) \as 0$ for $p = q$. 

\bigskip

\begin{figure}[h]
   \includegraphics[width=9.7cm,height=6cm,keepaspectratio=true]{./Figures/fig-1.eps}
   \caption{\label{fig:Hilbert-metric} 
   The Hilbert metric $\dC$}
\end{figure}

\bigskip

The metric space $(\cC , \dC)$ thus obtained is a complete non-compact geodesic metric space 
whose topology is the one induced by the canonical topology of $\Rn{m}$ and in which the affine open segments 
joining two points of the boundary $\bC$ are geodesic lines. 
It is to be mentioned here that in general the affine segment between two points in $\cC$ 
may \emph{not} be the \emph{unique} geodesic joining these points (for example, if $\cC$ is a square). 
Nevertheless, this uniqueness holds whenever $\cC$ is \emph{strictly} convex. 

\bigskip
   
Moreover, the distance function $\dC$ is associated with the Finsler metric 
$\FC$ on $\cC$ given, for any $p \in \cC$ and any $v \in T_{\! p}\cC \equiv \Rn{m}$ 
(the tangent vector space to $\cC \.$ at $p$), by 
$$
\FC(p , v) 
\as 
\frac{1}{2} \! \left( \frac{1}{t^{-}} + \frac{1}{t^{+}} \right) \quad \mbox{for} \quad v \ \neq \ 0~,
$$ 
where $t^{- \!} = t_{\cC}^{- \!}(p , v)$ and $t^{+ \!} = t_{\cC}^{+ \!}(p , v)$ 
are the \emph{unique positive} numbers satisfying $p - t^{-} v \in \bC$ and $p + t^{+} v \in \bC$, 
and $\FC(p , v) \as 0$ for $v = 0$. 

\bigskip

\begin{remark*} 
For $p \in \cC$ and $v \in T_{\! p}\cC \equiv \Rn{m}$ with $v \neq 0$, we will define 
$p^{- \!} = p_{\cC}^{- \!}(p , v) \as p - t_{\cC}^{- \!}(p , v) v$ 
and $p^{+ \!} = p_{\cC}^{+ \!}(p , v) \as p + t_{\cC}^{+ \!}(p , v) v$ (see Figure~\ref{fig:Finsler-metric}). 
Then, given any arbitrary norm $\norm{\, \cdot \,}$ on $\Rn{m}$, we can write 
$$
\FC(p , v) 
\ = \ 
\frac{1}{2} \! \norm{v} \!\! \left( \frac{1}{\norm{p - p^{-}}} + \frac{1}{\norm{p - p^{+}}} \right).
$$ 
\end{remark*}

\bigskip

\begin{figure}[h]
   \includegraphics[width=9.7cm,height=6cm,keepaspectratio=true]{./Figures/fig-2.eps}
   \caption{\label{fig:Finsler-metric} 
   The Finsler metric $\FC$}
\end{figure}

\bigskip

Finally, let $\vol$ be the canonical Lebesgue measure on $\Rn{m}$ and define $\om_{m} \as \vol(\Bn{m})$. 

\medskip

For $p \in \cC$, let $B_{\. \cC \.}(p) \as \{v \in \Rn{m} \st \FC(p , v) < 1\}$ be the unit open ball 
with respect to the norm $\FC(p , \cdot)$ on $T_{\! p}\cC \equiv \Rn{m}$. \\ 
The measure $\mC$ on $\cC$ associated with the Finsler metric $\FC$ 
is then defined, for any Borel set $A \inc \cC$, by 
$$
\mC(A) 
\as \!\!\. 
\int{A}{~}{\frac{\om_{m}}{\vol(B_{\. \cC \.}(p))}}{\vol(p)}
$$ 
and will be called the \emph{Hilbert measure} associated with $(\cC , \dC)$. 

\bigskip

\begin{remark*} 
The Borel measure $\mC$ is the classical Busemann measure of the Finsler space $(\cC , \FC)$ 
and corresponds to the Hausdorff measure of the metric space $(\cC , \dC)$ 
(see \cite[page~199, Example~5.5.13]{BBI01}). 
\end{remark*}

\bigskip

Thanks to this measure, we can make use of a concept of fundamental importance, 
the \emph{volume growth entropy}, which is attached to any metric space. 
Very often, this notion is introduced for cocompact metric spaces 
and is defined as follows in Hibert geometry. 

\medskip

Let $(\cC , \dC)$ be a Hilbert domain in $\Rn{m}$ admitting a cocompact group of isometries, 
for which we may assume $0 \in \cC$ since translations in $\Rn{m}$ preserve the cross ratio. 

\medskip

If for any $R > 0$ we denote by $\BC(0 , R) \as \{ p \in \cC \st \dC(0 , p) < R \}$ 
the open ball of radius $R$ about $0$ in $(\cC , \dC)$, then the volume growth entropy of $\dC$ writes 
$$
h(\cC) \: \as \!\! \lim_{R \goes +\infty} \frac{1}{R} \ln{\! [\mC(\BC(0 , R))]}~.
$$ 

\bigskip

Now, when we drop cocompactness, this limit still exists in the case 
when the boundary $\bC$ of $\cC$ is strongly convex (see \cite{ColVero04}) 
or in the case when $\cC$ is a polytope (see \cite{Ver09}), 
but it is not known whether this is true in general 
as stated in \cite[question raised in section~2.5]{Cra09}. 

\medskip

Therefore, the main goal of this paper is to answer to this question 
and to show that the answer is \emph{negative}. 

\bigskip

\begin{maintheorem} 
   There exists a Hilbert domain $(\cC , \dC)$ in $\Rn{2}$ with $0 \in \cC$ that satisfies 
   $$
   \limsup_{R \goes +\infty} \frac{1}{R} \ln{\! [\mC(\BC(0 , R))]} \ = \ 1~,
   $$ 
   and
   $$
   \liminf_{R \goes +\infty} \frac{1}{R} \ln{\! [\mC(\BC(0 , R))]} \ = \ 0~.
   $$ 
\end{maintheorem}

\bigskip

The proof of this theorem will be given in the third section by constructing 
an explicit example which is a convex `polygon' with \emph{infinitely} many vertices 
having an accumulation point around which the boundary of the `polygon' 
strongly looks like a circle. 

\medskip

The intuitive idea behind this construction is that, depending on where we are located in the `polygon', 
its boundary may look like the one of a usual polygon 
--- and hence the volume growth entropy behaves as if it were vanishing --- 
(this corresponds to the $\liminf$ part in the theorem) 
or like a small portion of a circle (around the accumulation point) 
--- and hence the volume growth entropy behaves as if it were positive --- 
(this corresponds to the $\limsup$ part in the theorem). 

\bigskip

On the other hand, using such `polygons' with infinitely many vertices 
and considering the same techniques as in the proof of the main theorem, 
we show that there are \emph{other} Hilbert domains in the plane than polygonal ones 
whose volume growth entropy is zero. 
This is stated in Theorem~\ref{thm:zero-entropy}. 

\bigskip

For further information about Hilbert geometry, we refer to \cite{Bus55,BusKel53,Egl97,Gol88,Ver05} 
and the excellent introduction \cite{Soc00} by Soci\'{e}-M\'{e}thou. 

\medskip

About the importance of volume growth and topological entropies in Hilbert geometry, 
we may have a look at the interesting work \cite{Cra09} by Crampon and the references therein. 

\bigskip
\bigskip
\bigskip


\section{Preliminaries} \label{sec:preliminaries} 

This section is devoted to listing the key ingredients we will need in the present work. 

\bigskip

\begin{notations} 

From now on, the canonical Euclidean norm on $\Rn{2}$ will be denoted by $\norm{\, \cdot \,}$. 

\medskip

On the other hand, for any three distinct points $a$, $b$ and $c$ in $\Rn{2}$, 
we will denoted by $a b c$ their open convex hull (open triangle), 
and by $\sphericalangle(b a c)$ the sector defined as the convex hull of the union 
of the half-lines $a + \RR_{+} \vect{a b}$ and $a + \RR_{+} \vect{a c}$. 

\end{notations}

\bigskip

The first ingredient, whose proof can be found for example in \cite[page~69]{Vin93}, 
is classic and concerns the hyperbolic plane given here by its Klein model $(\Bn{2} , \dB)$. 

\bigskip

\begin{proposition} \label{prop:Klein-disk} 
   We have 
   
   \smallskip
   
   \begin{enumerate}
      \item $\dB(0 , p) = \atanh{\! (\norm{p})}$ for any $p \in \Bn{2}$, and 
      
      \bigskip
      
      \item $\disp \mB(\BB(0 , R)) = \frac{\pi}{2} \sinh^{\. 2 \.}{\! (\. R)}$ for any $R > 0$. 
   \end{enumerate} 
\end{proposition}

\bigskip

The two following results have been established in \cite[Proposition~5 and Proposition~6]{CVV04}. 

\bigskip

\begin{proposition} \label{prop:comparison} 
   If $(\cC , \dC)$ and $(\cD , \dD)$ are Hilbert domains in $\Rn{m}$ 
   satisfying $\cC \inc \cD$, then the following properties are true: 
   
   \smallskip
   
   \begin{enumerate}
      \item Given any two distinct points $p , q \in \cC$, we have $\dC(p , q) \geq \dD(p , q)$ 
      with equality if and only if 
      $(p + \RR_{+} \vect{p q}) \cap \bC \, = \, (p + \RR_{+} \vect{p q}) \cap \bD$ and 
      $(p + \RR_{-} \vect{p q}) \cap \bC \, = \, (p + \RR_{-} \vect{p q}) \cap \bD$ hold. 
      
      \bigskip
      
      \item For any $p \in \cC$, we have $\, \vol(B_{\. \cC \.}(p)) \leq \vol(B_{\. \cD \.}(p))$. 
      
      \bigskip
      
      \item For any Borel set $A \inc \cC$, we have $\, \mC(A) \geq \mD(A)$. 
   \end{enumerate} 
\end{proposition}

\bigskip

\begin{proposition} \label{prop:square} 
   If $\cQ \as ({- \. 1} , 1) \cart ({- \. 1} , 1) \inc \Rn{2}$ denotes the standard open square, 
   then for any $p = (x , y) \in \cQ$ we have 
   $$
   2 (1 - x^{2}) \. (1 - y^{2}) \ \leq \ \vol(B_{\. \cQ \.}(p)) \ \leq \ 4 (1 - x^{2}) \. (1 - y^{2})~.
   $$ 
\end{proposition}

\medskip 

The last ingredient can be found in \cite[Proof of Theorem~12]{ColVer06}. 

\bigskip

\begin{proposition} \label{prop:vol-inequality} 
   Given any Hilbert domain $(\cC , \dC)$ in $\Rn{m}$ satifying $0 \in \cC$, we have 
   $$
   \vol(\BC(0 , R)) \ \leq \ e^{8 R} \!\. \times \! \vol(B_{\. \cC \.}(p))
   $$ 
   for all $R > 0$ and $p \in \BC(0 , R)$. 
\end{proposition}

\bigskip

We now give two technical lemmas which will be used for proving both 
Theorem~\ref{thm:no-limit} and Theorem~\ref{thm:zero-entropy}. 

\bigskip

\begin{lemma} \label{lem:vol-upper-bound-triangle} 
   Let $(\cC , \dC)$ be a Hilbert domain in $\Rn{2}$ with $0 \in \cC$, 
   and let $P$, $Q$ be distinct points in $\Rn{2}$ such that the affine segments $[P , Q]$ 
   and $[{- \. P} , -Q]$ are contained in the boundary $\bC$. 
   
   \medskip
   
   If $\cT$ is the open quadrilateral in $\Rn{2}$ defined as the open convex hull 
   of $P$, $Q$, $-P$ and $-Q$, then for any $R > 0$ we have 
   
   \smallskip
   
   \begin{enumerate}
      \item $\BC(0 , R) \cap P 0 Q \, = \, \BT(0 , R) \cap P 0 Q$, and 
      
      \bigskip
      
      \item $\mC\big( \BC(0 , R) \cap P 0 Q \big) \leq 2 \pi R^{2} \!$. 
   \end{enumerate} 
\end{lemma}

\bigskip

\begin{figure}[h]
   \includegraphics[width=11cm,height=11cm,keepaspectratio=true]{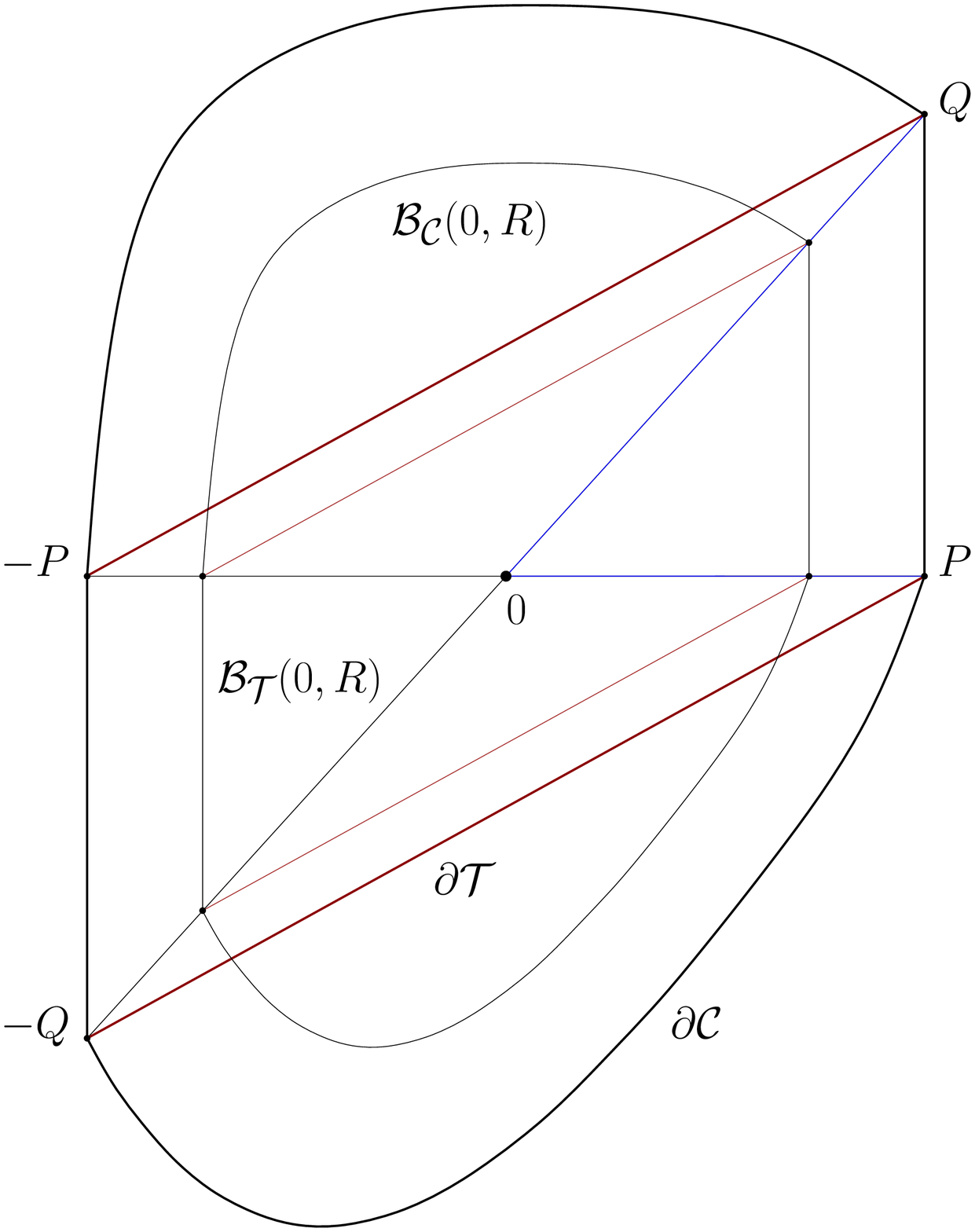}
   \caption{\label{fig:vol-upper-bound-triangle} 
   Comparing $\BC(0 , R) \cap P 0 Q$ and $\BT(0 , R) \cap P 0 Q$}
\end{figure}

\bigskip

\begin{proof}[Proof (see Figure~\ref{fig:vol-upper-bound-triangle})]~\\ 
\textbullet \ \textsf{Point~1.} 
The equality case in Point~1 of Proposition~\ref{prop:comparison} proves that 
any point $p \in \. P 0 Q$ satisfies $\dC(0 , p) = \dT(0 , p)$, 
and hence we get $\SC(0 , R) \cap P 0 Q = \ST(0 , R) \cap P 0 Q$. 

\smallskip

Then, writing $\disp \BC(0 , R) = \!\!\! \bigcup_{\. r \in [0 , R)} \!\!\!\. \SC(0 , r)$, 
we have $\BC(0 , R) \cap P 0 Q = \BT(0 , R) \cap P 0 Q$. 

\medskip

\textbullet \ \textsf{Point~2.} 
The previous point implies 
$$
\BC(0 , R) \cap P 0 Q 
\ \inc \ 
\BT(0 , R) 
\ \inc \ 
\cT 
\ \inc \ 
\cC~,
$$ 
and hence 
\begin{equation} \label{equ:vol-comparison} 
   \mC\big( \BC(0 , R) \cap P 0 Q \big) 
   \ \leq \ 
   \mT(\BT(0 , R))~. 
\end{equation} 

\smallskip

Now, if $f$ denotes the unique linear transformation of $\Rn{2}$ such that 
$f(P) = (1 , -1)$ and $f(Q) = (1 , 1)$, we have 
$$
f(\cT) \ = \ \cQ \as (-1 , 1) \cart (-1 , 1) \inc \Rn{2} \quad \mbox{(standard open square)}~.
$$ 

\smallskip

The cross ratio being preserved by the linear group $\GL(\Rn{2 \,})$, the map $f$ induces an isometry 
between the metric spaces $(\cT , \dT)$ and $(\cQ , \dQ)$ with $f(0) = 0$, 
and thus we obtain $\mT(\BT(0 , R)) = \mQ(\BQ(0 , R))$. 

\smallskip

But Proposition~\ref{prop:square} yields 
\begin{eqnarray*} 
   \mQ(\BQ(0 , R)) 
   & = & 
   8 \!\! \int{0}{\tanh{\! (\. R)} \,}
   {\left( \int{0}{x}{\frac{\pi}{\vol(B_{\. \cQ \.}(x , y))}}{y} \!\! \right) \!\!}{x} \\ 
   & \leq & 
   4 \!\! \int{0}{\tanh{\! (\. R)} \,}
   {\left( \int{0}{x}{\frac{\pi}{(1 - x^{2}) \. (1 - y^{2})}}{y} \!\! \right) \!\!}{x} 
   \ = \ 
   2 \pi R^{2}, 
\end{eqnarray*} 
which gives $\mC\big( \BC(0 , R) \cap P 0 Q \big) \leq 2 \pi R^{2}$ 
from Equation~\ref{equ:vol-comparison}. 
\end{proof}

\bigskip

\begin{lemma} \label{lem:vol-upper-bound-sector} 
   Let $(\cC , \dC)$ be a Hilbert domain in $\Rn{2}$ which satisfies $0 \in \cC \inc \Bn{2}$, 
   and let $A$, $B$ be two distinct points in $\Sn{1}$. 
   
   \medskip
   
   Then for any $R > 0$ we have 
   $$
   \mC\big( \BC(0 , R) \cap \sphericalangle(A 0 B) \big) 
   \ \leq \ 
   \frac{\pi e^{8 R}}{\vol(\BC(0 , 1))} \! \times \! \widehat{A 0 B} \. / 2~,
   $$ 
   where $\widehat{A 0 B}$ is the spherical distance between the vectors $\vect{0 A}$ and $\vect{0 B}$ 
   (\ie, the unique number $\t$ in $[0 , \pi]$ defined by 
   $\cos{\t} = \scal{\vect{0 A} / \! \norm{\vect{0 A}} \,}
   {\, \vect{0 B} \. / \! \norm{\vect{0 B}}} \! \in \! [-1 , 1]$, 
   where $\scal{\cdot \,}{\cdot}$ stands for the canonical Euclidean scalar product on $\Rn{2}$). 
\end{lemma}

\bigskip

\begin{proof}~\\ 
Since we have $1 / \vol(B_{\. \cC \.}(p)) \leq e^{8 R \!} / \vol(\BC(0 , R))$ 
for every $p \in \BC(0 , R)$ by Proposition~\ref{prop:vol-inequality}, one can write 
\begin{eqnarray*} 
   \mC\big( \BC(0 , R) \cap \sphericalangle(A 0 B) \big) 
   & \leq & 
   \frac{\pi e^{8 R}}{\vol(\BC(0 , R))} 
   \vol\big( \BC(0 , R) \cap \sphericalangle(A 0 B) \big) \\ 
   & \leq & 
   \frac{\pi e^{8 R}}{\vol(\BC(0 , 1))} \vol\big( \Bn{2} \cap \sphericalangle(A 0 B) \big) \\ 
   & & 
   (\mbox{noticing $\, \BC(0 , 1) \inc \BC(0 , R) \inc \Bn{2 \,}$}) \\ 
   & = & 
   \frac{\pi e^{8 R}}{\vol(\BC(0 , 1))} \! \times \! \pi \! \times \! \widehat{A 0 B} \. / \. (2 \pi) 
   \ = \ 
   \frac{\pi e^{8 R}}{\vol(\BC(0 , 1))} \! \times \! \widehat{A 0 B} \. / 2~, 
\end{eqnarray*} 
where we used $\vol(\Bn{2 \,}) = \pi$. 
\end{proof}

\bigskip
\bigskip
\bigskip


\section{Entropy may not be a limit} \label{sec:no-limit} 

We prove in this section the main result of this paper which states that the volume growth entropy 
for a Hilbert domain may \emph{not} be a limit. 
To this end, we will approximate a disc in the plane by an inscribed `polygonal' domain 
with infinitely many vertices that have two accumulation points around which the boundary 
of the `polygonal' domain looks very strongly like the boundary of the disc. 

\bigskip

Let $\seq{n_{k}}{k}{0}$ be the sequence of positive integers defined by 
$$
n_{0} \as 3 \qquad \mbox{and} \qquad \forall k \ \geq \ 0, \quad n_{k + 1} \ = \ 3^{n_{k}^{2}}~.
$$ 

It is increasing and satisfies $n_{k} \to +\infty$ as $k \to +\infty$. 

\medskip

Next, define the sequences $\seq{\a_{k}}{k}{0}$ and $\seq{\t_{k}}{k}{0}$ in $\RR$ by 
$\a_{k} \as 2 \pi / n_{k}$ together with 
$$
\t_{0} \as 0 
\qquad \mbox{and} \qquad 
\forall k \geq 1, \quad \t_{k} \as \!\! \sum_{\ell = 0}^{k - 1} \a_{\ell} 
\ = \ 
2 \pi \! \sum_{\ell = 0}^{k - 1} \frac{1}{n_{\ell}}~.
$$ 

\medskip

Finally, consider the sequence $\seq{M_{k}}{k}{0}$ and the family 
$(P_{k \!}(j))_{\! (k , j) 
\in 
\{ (k , j) \in \ZZ^{^{ \! 2}} 
\st k \, \geq \, 0 \ \mbox{\tiny and} \ 0 \, \leq \, j \, \leq \, n_{k} \}}$ 
of points in $\Sn{1}$ defined by 
$$
M_{k} \as (\cos{\! (\t_{k})} \, , \, \sin{\! (\t_{k})}) 
\qquad \mbox{and} \qquad 
P_{k \!}(j) 
\as 
(\cos{\! (\t_{k} + \a_{k} j \. / \. n_{k})} \ , \ \sin{\! (\t_{k} + \a_{k} j \. / \. n_{k})})~,
$$ 
and denote by $\cC$ the open convex hull in $\Rn{2}$ of the set 
$$
\{ P_{k \!}(j) , -P_{k \!}(j) 
\st k \geq 0 \, \ \mbox{and} \ \, 0 \leq j \leq n_{k} \}~.
$$ 

\bigskip

Then we get the following (see Figure~\ref{fig:no-limit}): 

\medskip

\begin{theorem} \label{thm:no-limit} 
   We have 
   
   \smallskip
   
   \begin{enumerate}
      \item $\disp h(\cC) 
      \as \! 
      \limsup_{R \goes +\infty} \frac{1}{R} \ln{\! [\mC(\BC(0 , R))]} = 1$, and 
      
      \bigskip
      
      \item $\disp \liminf_{R \goes +\infty} \frac{1}{R} \ln{\! [\mC(\BC(0 , R))]} = 0$. 
   \end{enumerate} 
\end{theorem}

\bigskip

\begin{figure}[h]
   \includegraphics[width=10cm,height=10cm,keepaspectratio=true]{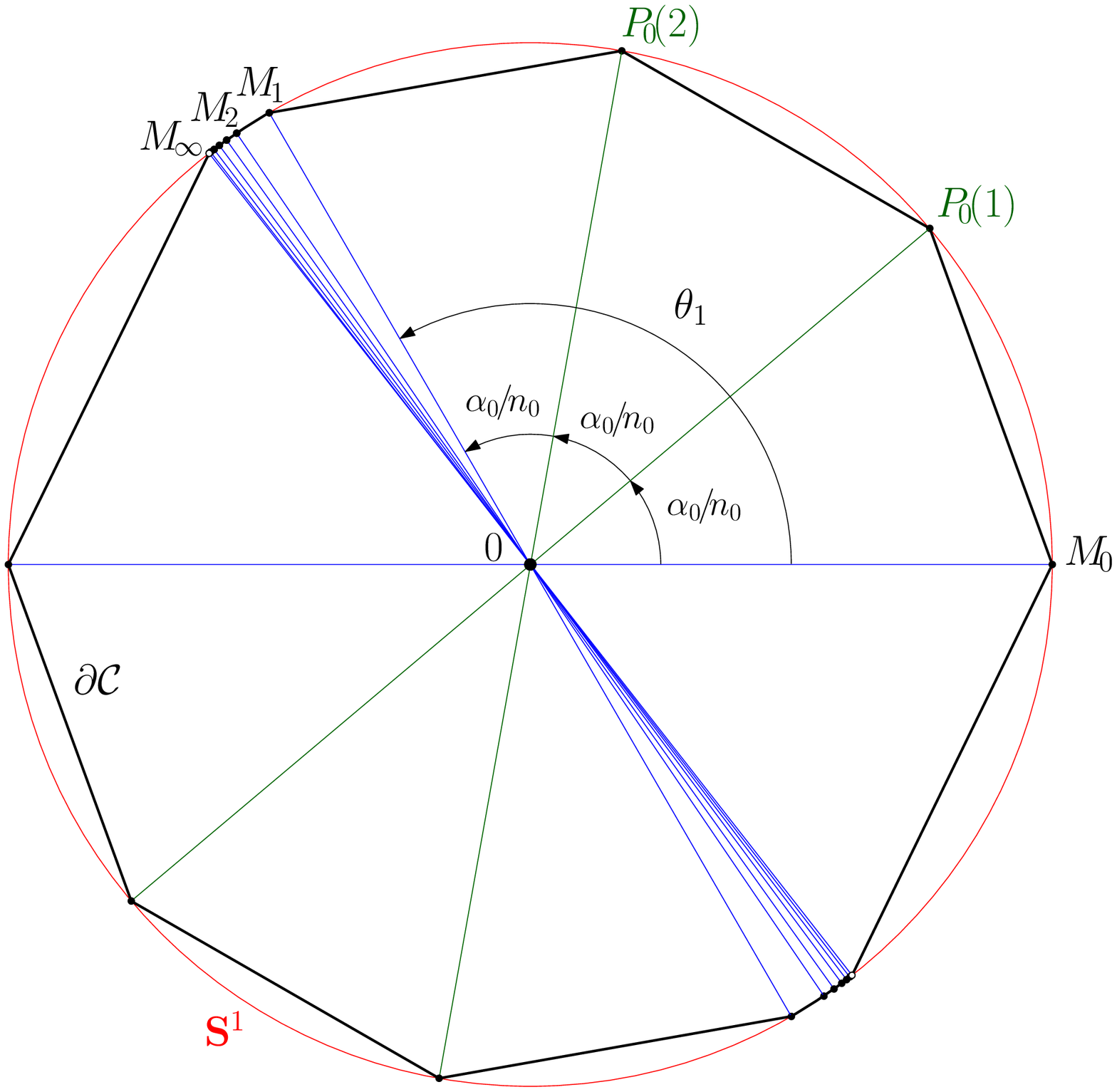}
   \caption{\label{fig:no-limit} 
   A Hilbert domain in the plane whose entropy is not a limit}
\end{figure}

\bigskip

\begin{remarks*} 
~ 
\begin{enumerate}[1)]
   \item For all $k \geq 0$, one has $P_{k \!}(0) = M_{k}$ and $P_{k \!}(n_{k}) = M_{k + 1}$. 
   
   \smallskip
   
   \item For all $\ell \geq 0$, we have $n_{\ell} \geq 3^{\ell + 1}$ 
   (by induction and using $9^{m} \geq m$ for all integer $m \geq 0$), 
   and hence the increasing sequence $\seq{\t_{k}}{k}{0}$ converges 
   to some real number $\t_{\infty}$ which satisfies $0 < \t_{\infty} < \pi$ 
   (since we have $\disp \sum_{\ell = 0}^{+\infty} 1 \. / \. 3^{\ell + 1} 
   \, = \, (1 \. / \. 3) \!\! \sum_{\ell = 0}^{+\infty} (1 \. / \. 3)^{\ell} \, = \, 1 \. / \. 2 \,$ 
   and $\, \seq{n_{\ell}}{\ell}{0} \, \neq \, \seq{3^{\ell + 1}}{\ell}{0}$). 
\end{enumerate} 
\end{remarks*}

\bigskip

In order to prove this theorem, we shall use two different sequences of balls about the origin. 
The first one corresponds to the sequence of radii $r_{k} \as \. \ln{\! (n_{k})}$ for $k \geq 0$ 
that makes look the balls like those in the Klein model $(\Bn{2} , \dB)$ as $k \to +\infty$, 
from which we get Point~1. The second one corresponds to the sequence of radii $R_{i} \as n_{i}$ for $i \geq 0$ 
that makes look the balls like those in a polygonal domain as $i \to +\infty$, leading to Point~2. 

\bigskip

\begin{proof}[Proof of Theorem~\ref{thm:no-limit}]~\\ 
\textbullet \ \textsf{Point~1.} 
Since we already have $h(\cC) \leq 1$ by \cite[Theorem~3.3]{BBV10}, let us prove $h(\cC) \geq 1$. 

\bigskip

Consider the sequence of positive numbers $\seq{r_{k}}{k}{0}$ defined by 
$r_{k} \as \. \ln{\! (n_{k})}$. 

\smallskip

Fix $k \geq 0$, and let $(p_{k \!}(j))_{0 \, \leq \, j \, \leq \, n_{k} - 1}$ 
be the sequence of points in $\Rn{2}$ defined by 
$$
p_{k \!}(j) \, \in \, [0 , P_{k \!}(j)] 
\qquad \mbox{and} \qquad 
\dC(0 , p_{k \!}(j)) \; = \ r_{k}~.
$$ 

\smallskip

Then fix $j \in \{ 0 , \ldots , n_{k} - 1 \}$, and let $P \as P_{k \!}(j)$, 
$Q \as P_{k \!}(j + 1)$, $p \as p_{k \!}(j)$ and $q \as q_{k \!}(j + 1)$ 
(see Figure~\ref{fig:limsup-positive}). 

\bigskip

\begin{figure}[h]
   \includegraphics[width=10cm,height=10cm,keepaspectratio=true]{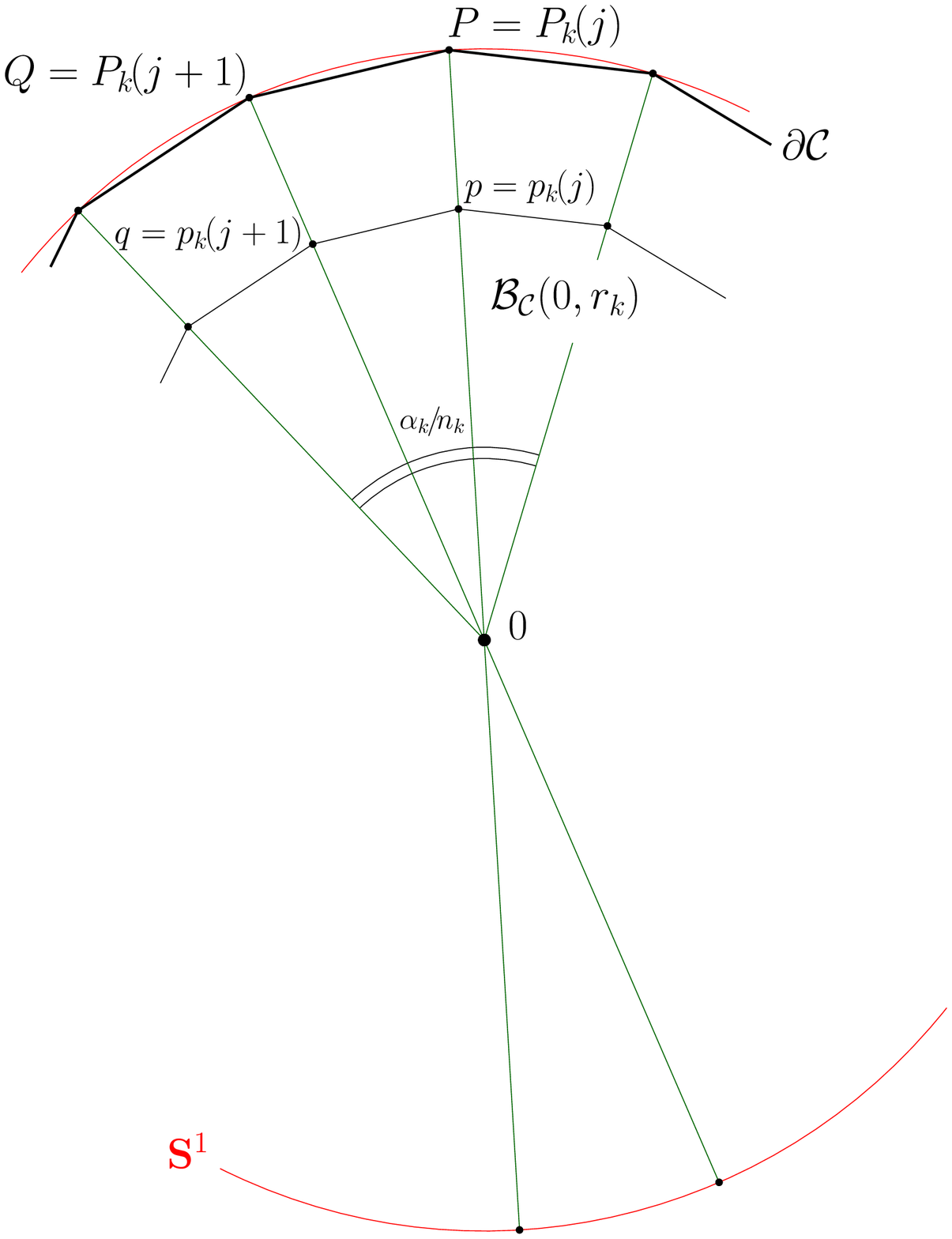}
   \caption{\label{fig:limsup-positive} 
   Showing $\disp \limsup_{k \goes +\infty} \frac{\ln{\! [\mC(\BC(0 , r_{k}))]}}{r_{k}} > 0$ 
   with $r_{k \.} := \ln{\! (n_{k})}$}
\end{figure}

\bigskip

First of all, since we have $\cC \inc \Bn{2}$, 
the equality case in Point~1 of Proposition~\ref{prop:comparison} gives 
$$
r_{k} \ = \ \dC(0 , p) \ = \ \dB(0 , p) 
\qquad \mbox{and} \qquad 
r_{k} \ = \ \dC(0 , q) \ = \ \dB(0 , q)~,
$$ 
which implies 
$$
\norm{p} \ = \ \norm{q} \ = \ \tanh{\! (r_{k})}
$$ 
by Point~1 of Proposition~\ref{prop:Klein-disk}. 

\smallskip

Therefore, if $a$ denotes the midle point of $p$ and $q$, we get 
$$
\norm{a} \ = \ \tanh{\! (r_{k})} \cos{\! [\a_{k} / \. (2 n_{k})]} 
\ = \ 
\tanh{\! (r_{k})} \cos{\! (\pi \. / \. n_{k}^{2})}~.
$$ 

\smallskip

Defining $\r_{k} \as \dB(0 , a)$ and using again Point~1 of Proposition~\ref{prop:comparison} 
together with the formula $\disp \tanh{\! (\ln{x})} = \frac{x^{2} - 1}{x^{2} + 1}$ 
which holds for any $x > 0$, one can write 
\begin{equation} \label{equ:radius} 
   \begin{split} 
      1 - \tanh{\! (\r_{k})} \ = \ 1 - \tanh{\! (r_{k})} \cos{\! (\pi \. / \. n_{k}^{2})} 
      & \ = \ 
      1 - \frac{n_{k}^{2} - 1}{n_{k}^{2} + 1} \cos{\! (\pi \. / \. n_{k}^{2})} \\ 
      & \ = \ 
      2 \. / \. n_{k}^{2} + \circ \. (1 \. / \. n_{k}^{2}) \ \sim \ 2 \. / \. n_{k}^{2} 
      \quad \mbox{as} \quad k \to +\infty~. 
   \end{split} 
\end{equation} 

\smallskip

On the other hand, the inclusions 
$$
\BB(0 , \r_{k}) \cap P 0 Q \ \inc \ p 0 q \ \inc \ \cC \ \inc \ \Bn{2}
$$ 
yield 
$$
\mC(p 0 q) \ \geq \ \mB(p 0 q) 
\ \geq \ 
\mB(\BB(0 , \r_{k}) \cap P 0 Q) 
\ = \ 
\frac{\a_{k} / \. n_{k}}{2 \pi} \! \times \! \mB(\BB(0 , \r_{k}))
$$ 
by Point~3 of Proposition~\ref{prop:comparison} 
and since Euclidean rotations induce isometries of $(\Bn{2} , \dB)$, 
from which one obtains 
\begin{equation} \label{equ:vol-p0q-1} 
   \mC(p 0 q) \ \geq \ \frac{\pi}{2} \sinh^{\. 2 \.}{\! (\. \r_{k})} \. / \. n_{k}^{2} 
\end{equation} 
by Point~2 in Proposition~\ref{prop:Klein-disk}. 

\smallskip

Now, since we have 
$\sinh^{\. 2 \.}{\! (x)} = \tanh^{\. 2 \.}{\! (x)} \. / \. (1 - \tanh^{\. 2 \.}{\! (x)})$ 
for all $x \in \RR$, Equation~\ref{equ:radius} implies 
$$
\frac{\pi}{2} \sinh^{\. 2 \.}{\! (\. \r_{k})} \. / \. n_{k}^{2} \ \sim \ \pi \. / 8 
\quad \mbox{as} \quad k \to +\infty~, 
$$ 
from which Equation~\ref{equ:vol-p0q-1} insures the existence of an integer $k_{0} \geq 0$ such that 
$$
\mC(p_{k \!}(j) 0 p_{k \!}(j + 1)) \ = \ \mC(p 0 q) \ \geq \ 1 \. / \. 3
$$ 
holds for every $k \geq k_{0}$. 

\smallskip

We then get 
$$
\mC(\BC(0 , r_{k})) 
\ \geq 
\sum_{j = 0}^{n_{k} - 1} \mC(p_{k \!}(j) 0 p_{k \!}(j + 1)) 
\ \geq \ 
(1 \. / \. 3) n_{k}
$$ 
for each $k \geq k_{0}$ 
since we have $p_{k \!}(j) 0 p_{k \!}(j + 1) \inc \BC(0 , r_{k})$ 
for every $j \in \{ 0 , \ldots , n_{k} - 1 \}$ (indeed, balls of a Hilbert domain are convex), 
and hence 
$$
\frac{\ln{\! [\mC(\BC(0 , r_{k}))]}}{r_{k}} 
\ \geq \ 
\frac{\ln{\! [(1 \. / \. 3) n_{k}]}}{\ln{\! (n_{k})}}~,
$$ 
which yields $\disp \frac{\ln{\! [\mC(\BC(0 , r_{k}))]}}{r_{k}} \to 1$ as $k \to +\infty$. 

\smallskip

This gives the first point of Theorem~\ref{thm:no-limit}. 

\medskip

\textbullet \ \textsf{Point~2.} 
Consider the sequence of positive numbers $\seq{R_{i}}{i}{0}$ defined by $R_{i} \as n_{i}$. 

\smallskip

Fixing an integer $i \geq 0$, we can write the decomposition 
\begin{equation} \label{equ:ball-decomposition-1} 
   \begin{split} 
      \frac{1}{2} \mC(\BC(0 , R_{i})) 
      & \ = \ 
      \mC\big( \BC(0 , R_{i}) \cap -M_{\infty} 0 M_{0} \big) \\ 
      & \ + \, 
      \sum_{k = 0}^{i} \sum_{j = 0}^{n_{k} - 1} 
      \mC\big( \BC(0 , R_{i}) \cap P_{k \!}(j) 0 P_{k \!}(j + 1) \big) \\ 
      & \ + \ 
      \mC\big( \BC(0 , R_{i}) \cap \sphericalangle(M_{i + 1} 0 M_{\infty}) \big) 
   \end{split} 
\end{equation} 
with $M_{\infty} \as (\cos{\! (\t_{\infty})} \, , \, \sin{\! (\t_{\infty})}) \in \bC$ 
(recall that $\t_{\infty}$ is the limit of the sequence $\seqN{\t}{k}$: 
see the second remark following Theorem~\ref{thm:no-limit}). 

\bigskip

\textasteriskcentered \ \textsf{First step.} 
Here, we deal with the two first terms in Equation~\ref{equ:ball-decomposition-1}. 

\smallskip

For each $k \geq 0$ and $j \in \{ 0 , \ldots , n_{k} - 1 \}$, 
let $\cTk(j)$ be the open rectangle that is equal to the open convex hull in $\Rn{2}$ 
of $P_{k \!}(j)$, $-P_{k \!}(j)$, $P_{k \!}(j + 1)$ and $-P_{k \!}(j + 1)$. 

\smallskip

Then, by Lemma~\ref{lem:vol-upper-bound-triangle}, we have 
\begin{equation} \label{equ:vol-upper-bound-1-1} 
   \mC\big( \BC(0 , R_{i}) \cap P_{k \!}(j) 0 P_{k \!}(j + 1) \big) \ \leq \ 2 \pi R_{i}^{2} 
\end{equation} 
and 
\begin{equation} \label{equ:vol-upper-bound-1-2} 
   \mC\big( \BC(0 , R_{i}) \cap -M_{\infty} 0 M_{0} \big) \ \leq \ 2 \pi R_{i}^{2 \!}~. 
\end{equation} 

\bigskip

\textasteriskcentered \ \textsf{Second step.} 
Next, we focus on the third term in Equation~\ref{equ:ball-decomposition-1}. 

\smallskip

Lemma~\ref{lem:vol-upper-bound-sector} with 
$\disp \widehat{M_{i + 1} 0 M_{\infty}} = \t_{\infty} - \t_{i + 1} 
= \, 2 \pi \!\!\! \sum_{\ell = i + 1}^{+\infty} \!\! 1 \. / \. n_{\ell} $ implies 
\begin{equation} \label{equ:vol-upper-bound-1-3} 
   \mC\big( \BC(0 , R_{i}) \cap \sphericalangle(M_{i + 1} 0 M_{\infty}) \big) 
   \ \leq \ 
   \tau \! \sum_{\ell = i}^{+\infty} e^{8 R_{i} \!} \. / \. n_{\ell + 1}~, 
\end{equation} 
where $\tau \as \pi^{2} \!\. / \. \vol(\BC(0 , 1))$ is a positive constant. 

\smallskip

But for any $\ell \geq i$ we have 
$$
e^{8 R_{i} \!} \. / \. n_{\ell + 1} 
\ = \ e^{8 R_{i}} 3^{-n_{\ell}^{2}} 
\ \leq \ 3^{8 R_{i}} 3^{-n_{\ell}^{2}} 
\ = \ 3^{8 n_{i} - n_{\ell}^{2}} 
\ = \ 3^{-n_{\ell}^{2}(1 - 8 n_{i} / \. n_{\ell}^{2})}
$$ 
with $n_{i} / \. n_{\ell}^{2} \leq 1 \. / \. n_{i}$ from the monotone increasing 
of the sequence $\seq{n_{\ell}}{\ell}{0}$. 

\smallskip

Hence, since $1 \. / \. n_{i} \to 0$ as $i \to +\infty$, there exists an integer $i_{0} \geq 0$ 
such that for all $\ell \geq i$ one has $e^{8 R_{i} \!} \. / \. n_{\ell + 1} \leq 3^{-n_{\ell}^{2 \.} / 2}$ 
whenever $i \geq i_{0}$. 

\smallskip

Equation~\ref{equ:vol-upper-bound-1-3} then implies 
$$
\mC\big( \BC(0 , R_{i}) \cap \sphericalangle(M_{i + 1} 0 M_{\infty}) \big) 
\ \leq \ 
\tau \! \sum_{\ell = i}^{+\infty} 3^{-n_{\ell}^{2 \.} / 2} 
\ \leq \ 
\tau \! \sum_{\ell = i}^{+\infty} 3^{-\ell} 
\ = \ 
3^{-i + 1} \tau \. / 2
$$ 
for all $i \geq i_{0}$ (notice that we have $n_{\ell}^{2 \.} / 2 \geq 9^{\ell} \geq \ell$ 
for any $\ell \geq 0$: see the second remark following Theorem~\ref{thm:no-limit}). 

\smallskip

Now we have $3^{-i} \to 0$ as $i \to +\infty$, and thus there exists an integer $i_{1} \geq i_{0}$ 
such that for all $i \geq i_{1}$ one has 
\begin{equation} \label{equ:vol-upper-bound-1-4} 
   \mC\big( \BC(0 , R_{i}) \cap \sphericalangle(M_{i + 1} 0 M_{\infty}) \big) \ \leq \ 1~. 
\end{equation} 

\bigskip

\textasteriskcentered \ \textsf{Third step.} 
Combining Equations~\ref{equ:ball-decomposition-1}, \ref{equ:vol-upper-bound-1-1}, 
\ref{equ:vol-upper-bound-1-2} and~\ref{equ:vol-upper-bound-1-4}, we eventually get 
\begin{equation*} 
   \begin{split} 
      \mC(\BC(0 , R_{i})) 
      & \ \leq \ 
      4 \pi R_{i}^{2} \. + 4 \pi R_{i}^{2} \! \sum_{k = 0}^{i} n_{k} + 1 \\ 
      & \ \leq \ 
      4 \pi R_{i}^{2} \. + 4 \pi R_{i}^{2} (i + 1) n_{i} + 1 \\ 
      & \qquad 
      (\mbox{since the sequence $\seq{n_{k}}{k}{0}$ is non-decreasing}) \\ 
      & \ = \ 
      4 \pi R_{i}^{2} \. + 4 \pi (i + 1) R_{i}^{3} \. + 1 \\ 
      & \ \leq \ 
      12 \pi R_{i}^{4} 
   \end{split} 
\end{equation*} 
for all $i \geq i_{1}$ 
(since we have $R_{\ell} \as n_{\ell} \geq 3^{\ell + 1} \geq \ell + 1 \geq 1$ for every $\ell \geq 0$), 
and hence 
$$
\frac{\ln{\! [\mC(\BC(0 , 1))]}}{R_{i}} 
\ \leq \ 
\frac{\ln{\! [\mC(\BC(0 , R_{i}))]}}{R_{i}} 
\ \leq \ 
\frac{\ln{\! (12 \pi R_{i}^{4})}}{R_{i}}~,
$$ 
which yields $\disp \frac{\ln{\! [\mC(\BC(0 , R_{i}))]}}{R_{i}} \to 0$ as $i \to +\infty$. 

\smallskip

This proves the second point of Theorem~\ref{thm:no-limit}. 
\end{proof}

\bigskip

\begin{remark*} 
Considering the proof of Point~1 in Theorem~\ref{thm:no-limit}, 
we can observe that the conclusion $h(\cC) > 0$ we obtained is actually true 
for \emph{any} sequence of positive integers $\seq{n_{k}}{k}{0}$ 
provided the sequence $\seqN{\t}{k}$ converges to some real number $\t_{\infty}$ 
which satisfies $0 < \t_{\infty} < \pi$. 
\end{remark*}

\bigskip
\bigskip
\bigskip


\section{Non-polygonal domains may have zero entropy} \label{sec:non-polygonal-zero-entropy} 

In this section, we construct a Hilbert domain in the plane which is a `polygon' 
having infinitely many vertices and whose volume growth entropy is a limit that is equal to zero. 
This `polygon' is inscribed in a circle and its vertices have one accumulation point. 

\bigskip

Before giving our example, let us first recall the following result proved in~\cite{Ver09}: 

\medskip

\begin{theorem} \label{thm:polytope-entropy} 
   Given any open convex polytope $\cP$ in $\Rn{m}$ that contains the origin $0$, 
   the volume growth entropy of $\dP$ satisfies 
   $$
   h(\cP) \ = \lim_{R \goes +\infty} \frac{1}{R} \ln{\! [\mP(\BP(0 , R))]} \ = \ 0~.
   $$ 
\end{theorem}

\medskip 

\begin{remark*} 
Another --- but less direct --- proof of this theorem 
consists in saying that $(\cP , \dP)$ is Lipschitz equivalent to Euclidean plane as shown in \cite{Ber09} 
(and in \cite{CVV11} for the particular case when $n \as 2$), 
and hence $h(\cP) = 0$ since the volume growth entropy of any finite-dimensional normed vector space 
is equal to zero. 
\end{remark*}

\bigskip

Now, let us show that having zero volume growth entropy for a Hilbert domain in $\Rn{2}$ 
does \emph{not} mean being polygonal, that is, that the converse of Theorem~\ref{thm:polytope-entropy} 
is \emph{false}. 

\bigskip

Let $\seqN{P}{n}$ be the sequence of points in $\Sn{1}$ defined by 
$$
P_{n} \as (\cos{\! (2^{-n})} \, , \, \sin{\! (2^{-n})})~,
$$ 
and denote by $\cC$ the open convex hull in $\Rn{2}$ of the set 
$$
\{ P_{n} , -P_{n} \st n \in \NN \}~.
$$ 

\medskip 

Then we have (see Figure~\ref{fig:zero-entropy}) 

\medskip

\begin{theorem} \label{thm:zero-entropy} 
   The volume growth entropy of $\dC$ satisfies 
   $$
   h(\cC) \ = \lim_{R \goes +\infty} \frac{1}{R} \ln{\! [\mC(\BC(0 , R))]} \ = \ 0~.
   $$ 
\end{theorem}

\bigskip

\begin{remark*} 
More precisely, we will show in the proof of this result that the volume $\mC(\BC(0 , R))$ 
of the ball $\BC(0 , R)$ actually has at most the same growth as $R^{3}$ when $R$ goes to infinity. 
\end{remark*}

\bigskip

\begin{figure}[h]
   \includegraphics[width=9cm,height=9cm,keepaspectratio=true]{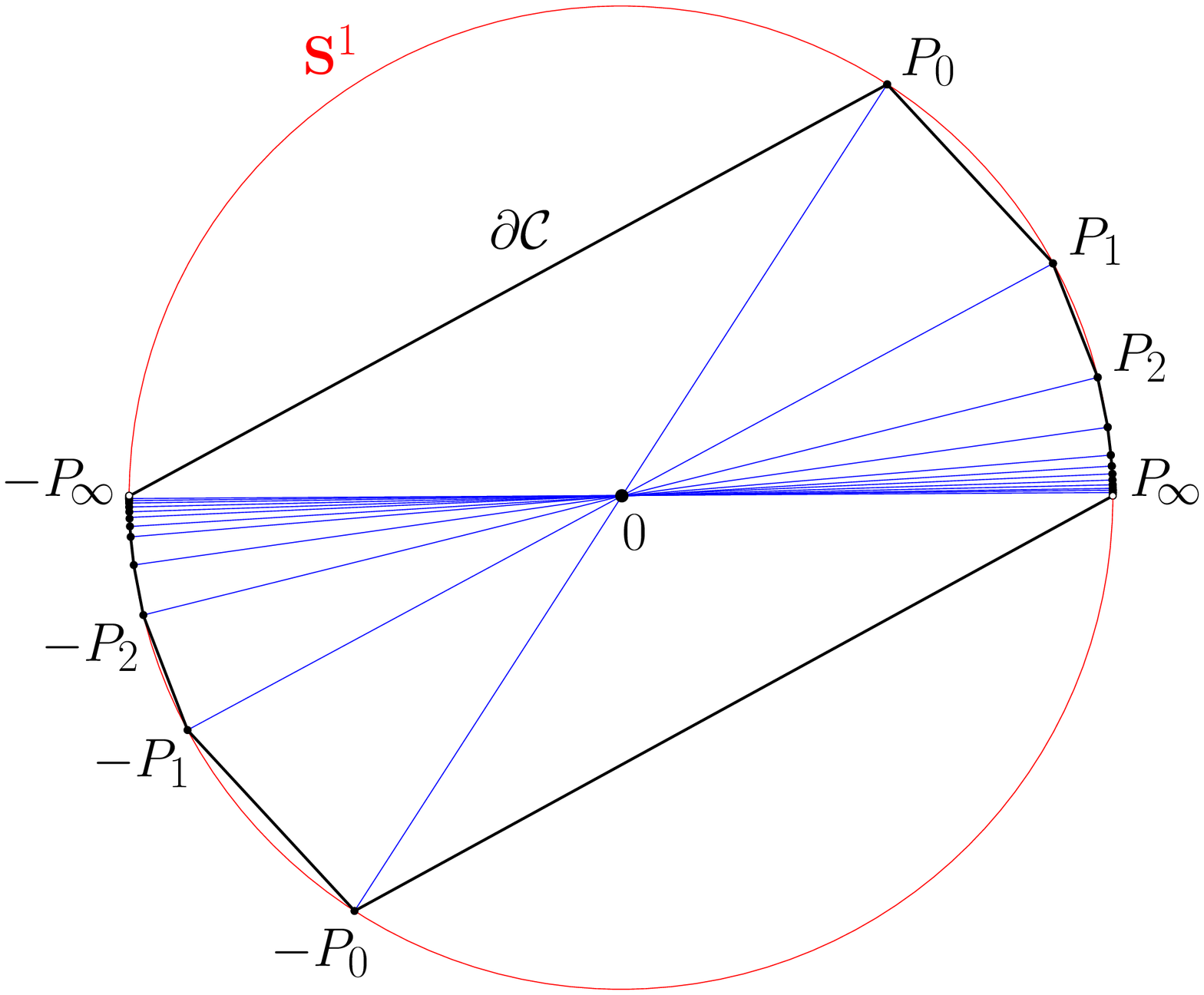}
   \caption{\label{fig:zero-entropy} 
   A non-polygonal Hilbert domain in the plane with zero entropy}
\end{figure}

\bigskip

\begin{proof}[Proof of Theorem~\ref{thm:zero-entropy}]~\\ 
Fixing an integer $n \geq 0$ and a number $R \geq 1$, 
we can use again the decomposition given by Equation~\ref{equ:ball-decomposition-1} 
in the proof of the second point of Theorem~\ref{thm:no-limit} and write 
\begin{equation} \label{equ:ball-decomposition-2} 
   \begin{split} 
      \frac{1}{2} \mC(\BC(0 , R)) 
      & \ = \ 
      \mC\big( \BC(0 , R) \cap -P_{\! \infty} 0 P_{0} \big) \\ 
      & \ + \, 
      \sum_{k = 0}^{n} \mC\big( \BC(0 , R) \cap P_{k} 0 P_{k + 1} \big) \\ 
      & \ + \ 
      \mC\big( \BC(0 , R) \cap \sphericalangle(P_{n + 1} 0 P_{\! \infty}) \big) 
   \end{split} 
\end{equation} 
with $\disp P_{\! \infty} \as (1 , 0) = \!\!\! \lim_{k \goes +\infty} \!\!\! P_{k} \in \bC$. 

\medskip

\textbullet \ \textsf{First step.} 
Here, we deal with the two first terms in Equation~\ref{equ:ball-decomposition-2}. 

\smallskip

For each $k \in \NN$, let $\cTk$ be the open rectangle 
that is equal to the open convex hull in $\Rn{2}$ of $P_{k}$, $-P_{k}$, $P_{k + 1}$ and $-P_{k + 1}$. 

\smallskip

Then, by Lemma~\ref{lem:vol-upper-bound-triangle}, we have 
\begin{equation} \label{equ:vol-upper-bound-2-1} 
   \mC\big( \BC(0 , R) \cap P_{k} 0 P_{k + 1} \big) \ \leq \ 2 \pi R^{2} 
\end{equation} 
and 
\begin{equation} \label{equ:vol-upper-bound-2-2} 
   \mC\big( \BC(0 , R) \cap -P_{\! \infty} 0 P_{0} \big) \ \leq \ 2 \pi R^{2}. 
\end{equation} 

\medskip

\textbullet \ \textsf{Second step.} 
Next, we focus on the third term in Equation~\ref{equ:ball-decomposition-2}. 

\smallskip

As in the second step of the proof of the second point of Theorem~\ref{thm:no-limit}, 
we use again Lemma~\ref{lem:vol-upper-bound-sector} with 
$\widehat{P_{n + 1} 0 P_{\! \infty}} = 2^{-(n + 1)} - 0 = 2^{-(n + 1)}$ to get 
\begin{equation} \label{equ:vol-upper-bound-2-3} 
   \mC\big( \BC(0 , R) \cap \sphericalangle(P_{n + 1} 0 P_{\! \infty}) \big) 
   \ \leq \ 
   \tau e^{8 R} \!\. \times \! 2^{-n}, 
\end{equation} 
where $\tau \as \pi \. / \. (4 \, \vol(\BC(0 , 1)))$ is a positive constant. 

\smallskip

So, if we choose $n \as [12 R] + 1$ (where $[ \, \cdot \, ]$ denotes the integer part), 
we have $e^{8 R} \!\. \times \! 2^{-n} \leq 1$, and hence Equation~\ref{equ:vol-upper-bound-2-3} implies 
\begin{equation} \label{equ:vol-upper-bound-2-4} 
   \mC\big( \BC(0 , R) \cap \sphericalangle(P_{n + 1} 0 P_{\! \infty}) \big) \ \leq \ \tau~. 
\end{equation} 

\medskip

\textbullet \ \textsf{Third step.} 
Combining Equations~\ref{equ:ball-decomposition-2}, \ref{equ:vol-upper-bound-2-1}, 
\ref{equ:vol-upper-bound-2-2} and~\ref{equ:vol-upper-bound-2-4}, 
we eventually obtain 
\begin{equation*} 
   \begin{split} 
      \mC(\BC(0 , R)) 
      & \ \leq \ 
      4 \pi R^{2} \! + 4 \pi (n + 1) R^{2} \! + \tau \\ 
      & \ \leq \ 
      4 \pi R^{2} \! + 4 \pi (12 R + 2) R^{2} \! + \tau \\ 
      & \qquad 
      (\mbox{since one has $\, n - 1 = [12 R] \leq 12 R$}) \\ 
      & \ \leq \ 
      (144 \pi + \tau) R^{3} 
   \end{split} 
\end{equation*} 
for any $R \geq 1$, and hence 
$$
\frac{\ln{\! [\mC(\BC(0 , 1))]}}{R} 
\ \leq \ 
\frac{\ln{\! [\mC(\BC(0 , R))]}}{R} 
\ \leq \ 
\frac{\ln{\! ((144 \pi + \tau) R^{3})}}{R}~,
$$ 
which yields $\disp \frac{\ln{\! [\mC(\BC(0 , R))]}}{R} \to 0$ as $R \to +\infty$. 

\smallskip

This proves Theorem~\ref{thm:zero-entropy}. 
\end{proof}

\bigskip
\bigskip
\bigskip


\bibliographystyle{acm}
\bibliography{math-biblio}

\end{document}